\documentclass[a4paper,10pt]{amsart}

\usepackage[utf8]{inputenc}
\usepackage{amsmath}
\usepackage{amssymb}
\usepackage{amsthm}
\usepackage{color}
\usepackage{enumitem}
\usepackage{hyperref}

\DeclareMathOperator{\st}{st}
\DeclareMathOperator{\Aut}{Aut}
\DeclareMathOperator{\rst}{rst}

\newcommand{\Z}{\mathbb{Z}}
\newcommand{\N}{\mathbb{N}}
\newcommand{\F}{\mathbb{F}}

\newtheorem*{thmA}{Theorem A}
\newtheorem*{thmB}{Theorem B}
\newtheorem*{thmC}{Theorem C}
\newtheorem*{thmD}{Theorem D}
\newtheorem{thm}{Theorem}[section]
\newtheorem{lem}[thm]{Lemma}
\newtheorem{prop}[thm]{Proposition}
\newtheorem{cor}[thm]{Corollary}

\theoremstyle{definition}
\newtheorem{defi}[thm]{Definition}

\theoremstyle{remark}
\newtheorem{remark}[thm]{Remark}

\title{On the congruence subgroup property for GGS-groups}
\author[G.A.\ Fern\'andez-Alcober]{Gustavo A.\ Fern\'andez-Alcober}
\address{Department of Mathematics\\ University of the Basque Country UPV/EHU\\
48080 Bilbao, Spain}
\email{gustavo.fernandez@ehu.eus}

\author[A.\ Garrido]{Alejandra Garrido}
\address{Heinrich-Heine-Universit\"{a}t D\"{u}sseldorf\\
Mathematisches Institut\\
Universit\"{a}tsstr. 1\\
40225\\
D\"{u}sseldorf, Germany}
\email{alejandra.garrido@uni-duesseldorf.de}

\author[J.\ Uria-Albizuri]{Jone Uria-Albizuri}
\address{Department of Mathematics\\ University of the Basque Country UPV/EHU\\
48080 Bilbao, Spain}
\email{jone.uria@ehu.eus}

\date{}
\thanks{G.A.\ Fern\'andez-Alcober and J.\ Uria-Albizuri acknowledge financial support from the Spanish Government, grants MTM2011-28229-C02 and
MTM2014-53810-C2-2-P, and from the Basque Government, grants IT753-13 and IT974-16.
J.\ Uria-Albizuri is also supported by the Basque Goverment predoctoral grant PRE-2014-1-347.
This article was finished while A.\ Garrido was a postdoctoral researcher at the Universit\'e de Gen\`eve, whose support and that of the Swiss National Science Foundation she gratefully acknowledges.}


\begin{document}

\begin{abstract}
We show that all GGS-groups with non-constant defining vector satisfy the congruence subgroup property. 
This provides, for every odd prime $p$, many examples of finitely generated, residually finite, non-torsion groups whose profinite completion is a pro-$p$ group, and among them we find torsion-free groups.
This answers a question of Barnea. 
On the other hand, we prove that the GGS-group with constant defining vector has an infinite congruence kernel and is not a branch group. 
\end{abstract}

\maketitle

\section{Introduction}

Groups of automorphisms of regular rooted trees have received considerable attention in the last few decades, motivated by the striking properties of some of the first examples studied.
The Grigorchuk $2$-groups \cite{grig} and the Gupta--Sidki $p$-groups \cite{guptasidki}, the most popular examples, are easily seen to be infinite finitely generated torsion groups, answering the General Burnside Problem.

There are now many generalizations in different directions of these initial examples: spinal groups, self-similar groups, branch groups, etc. 
One of the closest generalizations is the class of GGS-groups  (this stands for Grigorchuk--Gupta--Sidki, a term coined by Baumslag in \cite{Baumslag_TopicsCombinatorialGroupTheory}), whose properties have been studied by several authors \cite{Alc, Gri, Pervova_GGS, Pervova_CSP_GGS,Vovkivsky}.
The GGS-groups are finitely generated groups of automorphisms of the rooted $p$-regular tree $T$, where $p$ is an odd prime.
More precisely, to every non-zero vector $\mathbf{e}=(e_1,\dots,e_{p-1})$ with entries in $\mathbb{F}_p$ there corresponds a GGS-group $G=\langle a,b \rangle$, where $a$ is the rooted automorphism defined by the cycle $(1\; 2\; \dots\; p)$ and $b\in\st(1)$ is recursively defined by means of
\[
\psi(b)=(a^{e_1},a^{e_2},\dots,a^{e_{p-1}},b).
\]
(The definition of $\psi$, as well as the notation we use when working with automorphisms of $T$, can be found at the beginning of Section 2.)
For example, the Gupta--Sidki $p$-group corresponds to the vector $(1,-1,0,\dots,0)$. 
We say that $\mathbf{e}$ is symmetric if $e_i=e_{p-i}$ for $i=1,\ldots,p-1$.
Obviously, if the vectors $\mathbf{e}$ and $\mathbf{e'}$ are scalar multiples of each other, then they define the same GGS-group.
In particular, there is only one GGS-group with constant defining vector, which we denote by $\mathcal{G}$.
In dealing with $\mathcal{G}$, we will always assume that $\mathbf{e}=(1,\ldots,1)$.

The congruence subgroup property for subgroups of $\Aut T$ is defined by analogy with the same property  for arithmetic groups \cite{bass1967}, with the level stabilizers playing the role of congruences modulo ideals. 
More precisely, a subgroup $G$ of $\Aut T$ satisfies the congruence subgroup property
if each of its finite index subgroups contains some level stabilizer
$\st_G(n)=G\cap \st(n)$.
Taking the subgroups $\{\st_G(n) \mid n\in\N\}$ as a neighbourhood basis for the identity gives a topology on $G$, the congruence topology.
The completion $\overline G$ of $G$ with respect to this topology, which is called the \emph{congruence completion\/} of $G$, is a profinite group which is isomorphic to the closure of $G$ in $\Aut T$.
On the other hand, $G$ also embeds in its profinite completion $\widehat{G}$, and $\widehat{G}$ maps onto
$\overline{G}$. 
Now $G$ satisfying the congruence subgroup property is tantamount to the map
$\widehat{G}\to \overline{G}$ being an isomorphism.
The congruence subgroup problem asks whether this is the case and, if not, whether it is possible to determine the kernel of this map, which is called the \emph{congruence kernel\/} of $G$.

In \cite[Examples 10.1 and 10.2]{Gri}, Grigorchuk showed that the GGS-group corresponding to
$\mathbf{e}=(1,0,\ldots,0)$ is just infinite and satisfies the congruence subgroup property for $p\ge 5$, and that the same holds for all GGS-groups with $e_{p-3}=e_{p-2}=e_{p-1}=0$, provided that $p\ge 7$.
Vovkivsky proved that all torsion GGS-groups are just infinite \cite[Theorem 4]{Vovkivsky}, and then Pervova 
showed that torsion GGS-groups satisfy the congruence subgroup property \cite{Pervova_CSP_GGS}.
Observe that, according to \cite[Theorem 1]{Vovkivsky}, a GGS-group with defining vector $\mathbf{e}$ is torsion if and only if $e_1+\cdots+e_{p-1}=0$. 
As a consequence, most vectors of $\mathbb{F}_p^{p-1}$ define non-torsion GGS-groups.
Our first main result is the generalization of Pervova's theorem on the congruence subgroup property to all GGS-groups other than $\mathcal{G}$.

\begin{thmA}
 All GGS-groups with non-constant defining vector have the congruence subgroup property. 
\end{thmA}

Our proof is based on a general criterion of Bartholdi and Grigorchuk for a regular branch group to have the congruence subgroup property which, in particular, also yields that the groups in Theorem A are just infinite.
Also, it does not rely on the results of Pervova for torsion GGS-groups.

The  GGS-group with constant defining vector has a completely different behaviour.

\begin{thmB}
 The GGS-group $\mathcal{G}$ with constant defining vector has an infinite congruence kernel. 
\end{thmB}

It would be interesting to know whether the pro-$p$ completion of $\mathcal{G}$ coincides with its congruence completion, as well as to describe the congruence kernel of $\mathcal{G}$.
Previous work on the congruence subgroup problem for groups acting on rooted trees was done by Bartholdi, Siegenthaler and Zalesskii \cite{bartholdiCSP}, where they developed tools to determine the congruence kernel of branch groups.
However, these tools are not available to us, as the GGS-group with constant defining vector is not a branch group
(although it is weakly branch).
We also prove this fact, which had been mentioned  for the case $p=3$ in \cite[Proposition 7.3]{Bartholdi_Parabolic}.

\begin{thmC}
 The GGS-group $\mathcal{G}$ with constant defining vector is not a branch group.
\end{thmC}

In \cite{Barnea}, Barnea asked about the existence of infinite finitely generated residually finite non-torsion groups whose profinite completion is a pro-$p$ group, and also whether such groups may even be torsion-free.
Observe that Theorem A immediately answers Barnea's first question.
Indeed, for any GGS-group whose defining vector is non-constant, the profinite completion is the same as the congruence completion and, in particular, a pro-$p$ group.
As already mentioned, most of these groups are not torsion.
As for the second question, we also get a positive answer.

\begin{thmD}
Let $G$ be the GGS-group with defining vector $\mathbf{e}=(1,\ldots,1,\lambda)\in\F_p^{p-1}$, with
$\lambda\ne 1,2$.
Then $G'$ is an infinite, finitely generated, residually finite, and torsion-free group whose profinite completion is a
pro-$p$ group.
\end{thmD}

Actually, we get a general condition on $\mathbf{e}$, so that if the defining vector of $G$ satisfies this condition, then $G'$ will be torsion-free.

The paper is organised as follows.
In Section \ref{sec:non-constant} we give the proof of Theorem A. Section \ref{sec:constant} is devoted to the GGS-group $\mathcal{G}$ with constant defining vector and contains the proofs of Theorems B and C.
Finally, we prove Theorem D in Section \ref{sec:barnea}.

\section{GGS-groups with non-constant defining vector}\label{sec:non-constant}

In this section we prove Theorem A, i.e.\ that GGS-groups with a non-constant defining vector have the congruence subgroup property.
Before proceeding, we recall some facts about automorphisms of rooted trees and more specifically about GGS-groups.

Fix an odd prime $p$ and let $T$ be the regular rooted tree whose vertices are the elements of the monoid $X^*$ over an alphabet $X$ with $p$ elements, and two vertices $u$ and $v$ are joined by an edge if $v=ux$ or $u=vx$ for some
$x\in X$.
The set $L_n$ of all vertices of length $n$ is called the $n$th level of $T$, for every integer $n\ge 0$.
We denote by $\Aut T$ the group of automorphisms of $T$,  by $\st(v)$ the stabilizer of a vertex $v$ and 
by $\st(n)$ the stabilizer of all vertices in $L_n$.

Every vertex $v$ of $T$ is the root of a tree $T_v$ which is isomorphic to $T$, so we can define a map
$\psi_v:\st(v)\rightarrow \Aut T$ sending $g$ to its restriction $g_v$ to $T_v$.
Then we have an isomorphism 
\begin{align*}
\psi_n:\st(n) &\rightarrow \Aut T\times\overset{p^n}{\dots}\times \Aut T
\\ 
g &\mapsto (g_v)_{v\in L_n}.
\end{align*}
For simplicity, we write $\psi$ for $\psi_1$.
Observe also that $\psi_0$ is nothing but the identity map on $\Aut T$.

Now let $G$ be a subgroup of $\Aut T$.
We define $\st_G(n)$ and $\st_G(v)$ as the intersection with $G$ of the corresponding stabilizer in $\Aut T$.
Let $\rst_G(v)$ be the subgroup of all $g\in\st_G(n)$ such that $\psi_n(g)$ has all coordinates equal to $1$ except, possibly, at position $v$.
Then 
$$\rst_G(n) = \prod_{v\in L_n} \, \rst_G(v)$$
is the rigid stabilizer in $G$ of $L_n$. 
It is the largest subgroup of $\st_G(n)$ which maps onto a direct product under $\psi_n$.
If $G$ acts transitively on all levels of $T$, we say that $G$ is a branch group if $|G:\rst_G(n)|<\infty$ for all
$n$, and that $G$ is weakly branch if $\rst_G(n)\ne 1$ for all $n$.
Branch groups can be more generally defined when the rooted tree $T$ is not regular but level-homogeneous (see \cite[Section 5]{Gri}).
One can also speak about branch or weakly branch actions of a group on a rooted, level-homogeneous tree, by considering the induced group of automorphisms of the tree.

We say that $G$ is fractal if $\psi_v(\st_G(v))=G$ for all vertices $v$ of $T$; one can readily check that it suffices to require this condition for $v\in X$.
If $G$ is fractal, we say that $G$ is regular branch over a subgroup $K$
if $K\times \overset{p}{\cdots}\times K \subseteq \psi(K)$ and $K$ is of finite index in $G$.
This implies that
\[
K\times \overset{p^n}{\cdots}\times K \subseteq \psi_n(\st_G(n)) \subseteq G\times \overset{p^n}{\cdots}\times G
\]
and, as a consequence, $\rst_G(n)$ has finite index in $G$ for all $n\ge 1$.
Thus, if $G$ acts transitively on each $L_n$, it is a branch group. 
Removing the finite index constraint yields the definition of a weakly regular branch group, and if $G$ acts transitively on each $L_n$, then $G$ is in particular a weakly branch group.

As mentioned in the introduction, for every non-zero vector $\mathbf{e}=(e_1,\dots,e_{p-1})\in\mathbb{F}_p^{p-1}$ there exists a GGS-group $G=\langle a,b \rangle$.
Here, $a$ is the rooted automorphism corresponding to the cycle $\sigma=(1\; 2\; \dots\; p)$, that is,
$a(xv)=\sigma(x)v$ for all $x\in X$ and $v\in X^*$, and $b\in\st(1)$ is recursively defined by
$\psi(b)=(a^{e_1},a^{e_2},\dots,a^{e_{p-1}},b)$.
The definition of $b$ implies that $\psi_x(\st_G(1))=G$ for every $x\in X$, and consequently all GGS-groups are fractal.

\begin{remark}
\label{closure_in_sylow}
Consider the subgroup $\Aut_{\sigma} T$ of $\Aut T$ consisting of all automorphisms for which the permutation induced at every vertex of $T$ is a power of $\sigma$. 
Then $\Aut_{\sigma} T$ is a Sylow pro-$p$ subgroup of $\Aut T$ (see \cite[pp.\ 133--134]{Gri}).
By definition, every GGS-group $G$ is contained in $\Aut_{\sigma} T$, and so the congruence completion of $G$ is a pro-$p$ group.
\end{remark}

For many of the proofs, we heavily rely on \cite{Alc}, where a systematic approach to GGS-groups is given.
For the convenience of the reader, we collect here some of the results therein.

\begin{prop}
\cite[Theorem 3.2.1 and Corollary 3.2.5]{Alc}
\label{multi_lemmas}
Let $G$ be a GGS-group.
Then
\begin{enumerate}
\item
$\st_G(1)=\langle b\rangle^G=\langle b,b^a,\dots,b^{a^{p-1}}\rangle$;
\item
$\st_G(2)\leq G'\leq \st_G(1)$;
\item
$|G:G'|=p^2$ and $|G:\gamma_3(G)|=p^3$;
\item
$\st_G(2)\leq\gamma_3(G)$.
\end{enumerate}
\end{prop}

\begin{prop}
\cite[Lemmas 3.3.1 and 3.3.3]{Alc}
\label{branch}
Let $G$ be a GGS-group with non-constant defining vector.
Then
$$\psi(\gamma_3(\st_G(1)))=\gamma_3(G)\times\overset{p}{\cdots}\times\gamma_3(G).$$
If the defining vector is also non-symmetric, then
$$\psi(\st_G(1)')=G'\times\overset{p}{\cdots}\times G'.$$
\end{prop}

The above shows that all GGS-groups with non-constant defining vector are regular branch over $\gamma_3(G)$, and even over $G'$ when the defining vector is not symmetric.
As a consequence, they are branch groups.

Our proof that GGS-groups with a non-constant defining vector have the congruence subgroup property relies on the following result (see Proposition 3.8 of \cite{Bartholdi_Parabolic} and the proof of Theorem 4 of \cite{Gri}).

\begin{prop}
\label{criterion CSP}
Let $G\leq \Aut T$ be weakly regular branch over a subgroup $K$. 
If there exists $m\in\N$ such that $\st_G(m)\leq K'$, then $G$ has the congruence subgroup property and is just infinite.
More precisely, if $1\ne N\lhd G$ and $N\not\le \st_G(n)$ then $\st_G(n+m)\le N$.
\end{prop}

In the rest of the section we will show that, if $G$ is a GGS-group with non-symmetric defining vector, then
$G''$ contains some level stabilizer of $G$, and that the same property holds for non-constant symmetric defining vector, with $\gamma_3(G)'$ in the place of $G''$.
This will complete the proof of Theorem A.

\begin{lem}
\label{subdirect}
If $G$ is a GGS-group with non-constant defining vector, then $\psi(G')$ is a subdirect product of
$G\times\overset{p}{\cdots}\times G$.
\end{lem}

\begin{proof}
Since $b$ is defined by a non-constant vector, there exists $i\in\{1,\dots,p-1\}$ such that $e_i\neq e_{i+1}$.
Now observe that $[b,a]$ has $a^{-e_{1}}b$ in the first coordinate, while its conjugate $[b,a]^{a^{-i}}$ has the element
$a^{e_i-e_{i+1}}$. 
Since $G=\langle a^{-e_{1}}b, a^{e_{i}-e_{i+1}}\rangle$,  the projection of $\psi(G')$ on the first coordinate is the whole of $G$.
By conjugating by powers of $a$, we conclude that $\psi(G')$ is a subdirect product of
$G\times\overset{p}{\cdots}\times G$. 
\end{proof}

\begin{lem}
\label{subdirect2}
If $G$ is a GGS-group with non-constant symmetric defining vector, then $\psi(\gamma_3(G))$ is a subdirect product of $G\times\overset{p}{\cdots}\times G$.
\end{lem}

\begin{proof}
First of all, observe that if $p=3$ and the defining vector of $G$ is symmetric, then it must be constant.
Hence $p\geq 5$.
We have
\begin{multline*}
\psi(\left[b,a,a\right])
=
(b^{-1}a^{e_1}b^{-1}a^{e_{p-1}},a^{e_2-2e_1}b,a^{e_1-2e_2+e_3},\dots
\\
\dots,a^{e_{p-3}-2e_{p-2}+e_{p-1}},a^{-e_{p-1}}ba^{e_{p-2}-e_{p-1}}).
\end{multline*}
Since $\mathbf{e}$ is non-constant and symmetric, there exists $i\in\{1,\ldots,(p-3)/2\}$ such that $e_i\ne e_{i+1}$.
Let us choose $i$ as large as possible subject to that condition.
This choice, together with $e_{(p-1)/2}=e_{(p+1)/2}$, yields that $e_{i+1}=e_{i+2}$.
Consequently $e_{i}-2e_{i+1}+e_{i+2}=e_i-e_{i+1}\neq 0$, and the coordinate of $\psi([b,a,a])$ in position $i+2$ is a generator of $\langle a \rangle$.
Since we also have $a^{e_2-2e_1}b$ in the second position of $\psi([b,a,a])$, the result follows as in the proof of
Lemma \ref{subdirect}.
\end{proof}

We can now prove Theorem A.

\begin{thm}
Let $G$ be a GGS-group with non-constant defining vector. 
Then $G$ has the congruence subgroup property and is just infinite.
\end{thm}

\begin{proof}
By Propositions \ref{branch} and \ref{criterion CSP}, it suffices to show that $G''$ or $\gamma_3(G)'$ contain some level stabilizer, according as the defining vector $\mathbf{e}$ is non-symmetric or non-constant symmetric.

Assume first that $\mathbf{e}$ is non-symmetric.
We  have $\gamma_3(G)=\langle [g,a],[g,b]\mid g\in G'\rangle$.
By Proposition \ref{branch}, for each $g\in G'$ there exists $h\in\st_G(1)'$ such that $\psi(h)=(g,1,\dots,1)$.
On the other hand, by Lemma \ref{subdirect}, there exist $x,y\in G'$ such that $\psi(x)=(a,*,\dots,*)$ and
$\psi(y)=(b,*,\dots,*)$, where each $*$ denotes an undetermined element of $G$. 
Then $\psi([h,x])=([g,a],1,\dots,1)$ and $\psi([h,y])=([g,b],1,\dots,1)$ belong to $\psi(G'')$, and consequently
$\psi(G'')\geq \gamma_3(G)\times 1\times \dots\times 1$.
Upon conjugation by powers of $a$, we get $\psi(G'')\geq \gamma_3(G)\times\cdots\times \gamma_3(G)$.
Since $\st_G(2)\le \gamma_3(G)$ by (iv) of Proposition \ref{multi_lemmas}, we conclude that 
\[
\psi(G'') \geq \st_G(2) \times \cdots \times \st_G(2) = \psi(\st_G(3)),
\]
and $G''\geq\st_G(3)$, as desired.

Now we assume that $\mathbf{e}$ is non-constant symmetric.
Arguing as above, by combining Proposition \ref{branch} and Lemma \ref{subdirect2}, we get that
$\psi(\gamma_3(G)')\geq\gamma_4(G)\times\cdots\times\gamma_4(G)$.
If we show that $\st_G(3)\le\gamma_4(G)$ then $\st_G(4)\le \gamma_3(G)'$, and we are done.
By (iii) of Proposition \ref{multi_lemmas}, we have $|\st_G(1):\gamma_3(G)|=p^2$.
Hence $\st_G(1)'\le \gamma_3(G)$ and $\gamma_3(\st_G(1))\le \gamma_4(G)$.
Then
\begin{align*}
\psi(\gamma_4(G))
&\ge
\psi(\gamma_3(\st_G(1))) = \gamma_3(G) \times \cdots \times \gamma_3(G)
\\
&\ge
\st_G(2) \times \cdots \times \st_G(2) = \psi(\st_G(3)),
\end{align*}
by using Proposition \ref{branch}.
Thus $\st_G(3)\le \gamma_4(G)$, which completes the proof.
\end{proof}

% Grigorchuk had already proved in \cite[Examples 10.1 and 10.2]{Gri} that some examples of GGS-groups have the congruence subgroup property and are just infinite. 
% Vovkivsky showed in \cite[Theorem 4]{Vovkivsky} that torsion GGS-groups are just infinite (he does not mention the torsion assumption in the theorem but he uses it in the proof).
% 
% 
% 
% The above result had been proved for some special GGS-groups by Grigorchuk \cite[Examples 10.2]{Gri}, and for torsion GGS-groups by Vovkivsky \cite[Theorem 4]{Vovkivsky}.
% Notice that, even if the statement of Theorem 4 in \cite{Vovkivsky} does not mention the condition that the group should be torsion, that is necessary in order to apply Theorem 3 of the same paper.

\section{GGS-groups with constant defining vector}
\label{sec:constant}

In this section we prove that the GGS-group $\mathcal{G}$ with constant defining vector is not a branch group and
does not have the congruence subgroup property. 
Many of the ingredients for the proofs come from the analysis of this group developed in \cite[Section 4]{Alc}. 
Following that paper, we define $y_0 = ba^{-1}$ and $y_i= y_0^{a^i}$ for every integer
$i$ and note that $y_i^b=y_i^{aa^{-1}b}=y_{i+1}^{y_1}$.
An easy computation shows that $y_{p-1}y_{p-2}\dots y_1y_0=1$.

For the convenience of the reader, we state the following two lemmas from \cite{Alc}, which will be used in the sequel.

\begin{lem}{\cite[Lemma 4.2]{Alc}}\label{lemma4.2}
If $K=\langle y_0\rangle^{\mathcal{G}}$, then:
\begin{enumerate}
\item
$|\mathcal{G}:K|=p$, and as a consequence, $\st_{\mathcal{G}}(n)\le K$ for every $n\ge 2$.
\vspace{3pt}
\item
$K=\langle y_0,\dots,y_{p-1}\rangle$.
\item
$K'\times \overset{p}{\cdots}\times K' \leq \psi(K')\leq \psi(\mathcal{G}')\leq K\times\overset{p}{\cdots}\times K$.
In particular, $\mathcal{G}$ is a weakly regular branch group over $K'$.
\end{enumerate}
\end{lem}

\begin{lem}{\cite[Lemmas 4.3 and 4.4]{Alc}}
\label{lemma4.3}
For every $g\in K$ we have $gg^ag^{a^2}\dots g^{a^{p-1}}\in K'$.
Moreover, if $h\in K'$ with $\psi(h)=(h_1,\dots,h_p)$ then $h_p\dots h_1\in K'$.
\end{lem}

We start by determining the structure of the quotient $\mathcal{G}/K'$.
We need the following lemma.

\begin{lem}
The elements $y_0, \dots, y_{p-1}$ have infinite order.
\end{lem}

\begin{proof}
It suffices to prove the claim for $y_0$.
If the order of $y_0$ is finite, then it must be a power of $p$, say $p^n$, since
$\mathcal{G}$ is contained in a Sylow pro-$p$ subgroup of $\Aut T$.
Now,
\[
y_0^p = (ba^{-1})^p = b b^a \ldots b^{a^{p-1}}\in\st_{\mathcal{G}}(1),
\]
and
\[
\psi(y_0^p)=(aba^{p-2},a^2ba^{p-3},\dots,ba^{p-1})=(y_{p-1},y_{p-2},\dots,y_0).
\]
Thus the last coordinate of $\psi(y_0^{p^n})$ is $y_0^{p^{n-1}}$, which must be $1$.
This is a contradiction.
\end{proof}

\begin{prop}
\label{torsion-free}
The quotient group $\mathcal{G}/K'$ is isomorphic to the semidirect product
$$
P =  \langle d \rangle \ltimes \langle c_0,\dots,c_{p-2}\rangle 
\cong  C_p\ltimes (C_\infty\times\overset{p-1}{\dots}\times C_{\infty}),
$$
where $c_i^d=c_{i+1}$ for $i=0,\dots,p-3$ and $c_{p-2}^d=(c_0\dots c_{p-2})^{-1}$,
and the isomorphism maps $K/K'$ to the kernel of the semidirect product.
In particular, $K/K'$ is torsion-free.
\end{prop}

\begin{proof}
Taking into account that $y_i^a=y_{i+1}$ for all $i$ and that $y_{p-1}\ldots y_1y_0=1$, the assignments $c_i\mapsto y_iK'$ and $d\mapsto aK'$ define a surjective homomorphism $\alpha$ from $P$ to $\mathcal{G}/K'$, by Von Dyck's Theorem.
Thus we only need to show that $\ker\alpha=1$.
By way of contradiction, assume that the kernel of $\alpha$ contains an element
$w\ne 1$.

Put $C=\langle c_0,\ldots,c_{p-2} \rangle$, which is a free abelian group of rank $p-1$.
If $w\in P\smallsetminus C$ then $P=\langle w \rangle C$ and
$\alpha(P)=\alpha(C)=K/K'$, which is a contradiction.
Thus $w\in C$.
If $m$ is the order of the torsion subgroup of $C/\langle w \rangle$ then
$C^m\langle w \rangle/\langle w \rangle$ is free abelian of rank $p-2$.
Since $\alpha(C^m)=(K/K')^m$, it follows that the minimum number of generators of $(K/K')^m$ is $d((K/K')^m)\le p-2$.
Now, by \cite[Theorem 4.6]{Alc}, the quotient $\mathcal{G}/K'\st_{\mathcal{G}}(n)$ is a $p$-group of maximal class of order $p^{n+1}$ for every $n\ge 1$.
Let us choose $n=m(p-1)$.
Then the group $K/K'\st_{\mathcal{G}}(n)$  is homocyclic of rank $p-1$ and exponent $p^m>m$
(see \cite[Theorem 4.9]{Alc2} or \cite[Corollary 3.3.4]{lee-mck}).
Hence $d((K/K'\st_{\mathcal{G}}(n))^m)=p-1$, which is impossible since
$(K/K'\st_{\mathcal{G}}(n))^m$ is a homomorphic image of  $(K/K')^m$.
Thus $\ker\alpha=1$, as desired.
\end{proof}

We can now prove Theorem B.

\begin{thm}
The congruence kernel of the group $\mathcal{G}$ is infinite. 
In particular, $\mathcal{G}$ does not have the congruence subgroup property.
\end{thm}

\begin{proof}
Let $\widehat{\mathcal{G}}$ and $\overline{\mathcal{G}}$ be the profinite and congruence completions of $\mathcal{G}$, respectively, and let $C$ be the congruence kernel of $\mathcal{G}$, i.e.\ the kernel of the natural homomorphism from
$\widehat{\mathcal{G}}$ onto $\overline{\mathcal{G}}$. Recall from Remark \ref{closure_in_sylow} that $\overline{\mathcal{G}}$ is a pro-$p$ group.

Consider a prime $q$ other than $p$.
By Proposition \ref{torsion-free}, the factor group $\mathcal{G}/K'$ is a semidirect product with kernel $K/K'$ isomorphic to $C_\infty\times\overset{p-1}{\dots}\times C_{\infty}$ and complement isomorphic to $C_p$.
For every $n\in\N$, let $K_n$ be the normal subgroup of $\mathcal{G}$ defined by the condition $K_n/K'=(K/K')^{q^n}$.
Then $|\mathcal{G}:K_n|=pq^{n(p-1)}$.

A basic result in profinite group theory (see \cite[Proposition 3.2.2]{RibesZalesskii}) states that
there is a one-to-one correspondence $\Phi$ between the subgroups of $\mathcal{G}$ which are open in the profinite topology of $\mathcal{G}$ and the open subgroups of $\widehat{\mathcal{G}}$.
The map $\Phi$ takes an open subgroup $H\leq\mathcal{G}$ to the closure of $H$ in $\widehat{\mathcal{G}}$ 
(having identified $\mathcal{G}$ with its image in $\widehat{\mathcal{G}}$). 
Moreover, $\Phi$ preserves the indices between subgroups.
Thus, if $U_n=\Phi(K_n)$ then
\begin{equation}
\label{eqn}
pq^{n(p-1)}=|\mathcal{G}: K_n|=|\widehat{\mathcal{G}}: U_n|=|\widehat{\mathcal{G}}: U_nC| \, |U_nC:U_n|. 
\end{equation}
Now, $\widehat{\mathcal{G}}/ U_nC$ is a finite quotient of
\[
\widehat{\mathcal{G}}/C \cong \overline{\mathcal{G}} \cong \varprojlim_{n\in\N} \, \mathcal{G}/\st_{\mathcal{G}}(n),
\]
which is a pro-$p$ group.
Consequently $|\widehat{\mathcal{G}}: U_nC|$ is a power of $p$, and then by (\ref{eqn}),
\[
q^{n(p-1)}\mid |U_nC:U_n| = |C:U_n\cap C|
\]
for all $n\in\N$.
We conclude that $C$ is infinite, as desired.
\end{proof}

It is worth mentioning that the congruence kernel (and consequently also the congruence subgroup property) is independent of the branch action \cite[Theorem 1]{AleCSP} and indeed even of the weakly branch action \cite[Theorem 6.5]{Aletesis} that a group may have on a rooted, level-homogeneous tree.
In particular, the congruence kernel of $\mathcal{G}$ is also infinite for any other weakly branch action of $\mathcal{G}$.

Our next purpose is to prove Theorem C, i.e.\ that $\mathcal{G}$ is not a branch group.
This means that the techniques developed so far (in \cite{bartholdiCSP}) for the calculation of the congruence kernel of a subgroup of $\Aut T$ are not available in this case.
We need the following easy lemma.

\begin{lem}
\label{rist finite index}
Let $G$ be a subgroup of $\Aut T$, and assume that $|G:\rst_G(n)|$ is finite for some $n$.
If $H$ is a finite index subgroup of $G$, then $|H:\rst_H(n)|$ is also finite.
\end{lem}

\begin{proof}
Let $m$ be the index of $H$ in $G$.
Then
\begin{align*}
|\rst_G(n):\rst_H(n)|
&=
\Big | \prod_{u\in L_n} \rst_G(u):\prod_{u\in L_n} \rst_H(u) \Big |
\\
&=
\prod_{u\in L_n} \, |\rst_G(u):\rst_G(u)\cap H|
\le
m^{|L_n|}
\end{align*}
is finite, and the result follows.
\end{proof}

\begin{thm}
The group $\mathcal{G}$ is not a branch group.
\end{thm}

\begin{proof}
Let $L=\psi^{-1}(K'\times \cdots \times K')$.
By Lemma \ref{lemma4.2}, we have $L\subseteq \rst_{\mathcal{G}'}(1)$.
We claim that the equality holds.
To that purpose, we consider an element $g\in\rst_{\mathcal{G}'}(x)$, with $x\in X$, and we prove that $g\in L$.
By definition of rigid stabilizer of a vertex, all coordinates of $\psi(g)$ are equal to $1$,
except possibly the one corresponding to position $x$, say, $h$.
Observe that $h\in K$, since $\psi(\mathcal{G}')\subseteq K\times \cdots \times K$ by Lemma \ref{lemma4.2}.
If
\[
g^* = g g^a \ldots g^{a^{p-1}},
\]
then $g^* \in K'$ by Lemma \ref{lemma4.3}.
Now $\psi(g^*)=(h,\ldots,h)$ and, by applying the second part of Lemma \ref{lemma4.3}, we get $h^p\in K'$.
Since $h\in K$ and $K/K'$ is torsion-free by Proposition \ref{torsion-free}, it follows that $h\in K'$.
Thus $\psi(g)\in K'\times \cdots \times K'$, and $g\in L$, as desired.

Now assume by way of contradiction that $\mathcal{G}$ is a branch group.
Then $|\mathcal{G}:\rst_{\mathcal{G}}(1)|$ is finite, and by Lemma \ref{rist finite index} and the previous paragraph, $|\mathcal{G}':L|$ is also finite.
Now observe that $L\le K'$ by Lemma \ref{lemma4.2}.
Therefore the factor group $\mathcal{G}/K'$ is finite, which is a contradiction, according to
Proposition \ref{torsion-free}.
\end{proof}

\section{Barnea's questions}
\label{sec:barnea}

In \cite{Barnea}, Barnea posed the following two questions:
\begin{enumerate}
\item
Is there an infinite finitely generated residually finite non-torsion group such that its profinite completion is pro-$p$?
\item
Is there an infinite finitely generated residually finite torsion-free group such that its profinite completion is pro-$p$?
\end{enumerate}

According to Theorem A, the profinite completion of a GGS-group $G$ with non-constant defining vector is the same as its congruence completion.
Since $G$ lies in a Sylow pro-$p$ subgroup of $\Aut T$, the index $|G:\st_G(n)|$ is a
power of $p$ for all $n\ge 1$.
Thus the profinite completion of $G$ is a pro-$p$ group.
By considering non-constant vectors $\mathbf e$ with $e_1+\cdots+e_{p-1}\ne 0$, we get groups which answer in the positive Barnea's first question.
Note that the congruence subgroup property is hereditary for finite index subgroups.
Thus, in order to answer the second question, we consider the GGS-group with defining vector $\mathbf{e}=(1,\ldots,1,\lambda)$ with $\lambda\in\F_p\smallsetminus\{1,2\}$ and show that it is virtually torsion-free.
In the case $p=3$ and $\lambda=0$, this GGS-group is known as the Fabrykowski-Gupta group, and it was shown to be virtually torsion-free in \cite[Theorem 6.4]{Bartholdi_Parabolic}.

\vspace{8pt}

To start with, we identify which finite index subgroup should be shown to be torsion-free, using the following criterion.
% 
% 
% 
% To start with, we identify a possible torsion-free subgroup of finite index. 
% In fact, we have the following criterion  which we will use to determine when a regular branch group is virtually torsion-free.
% ............................... por que nos hace falta este criterio

\begin{prop}
 Let $G$ be a regular branch group over a subgroup $K$ and suppose that $G$ has the congruence subgroup property.
 If $\mathcal{P}$ is a property of groups which is hereditary for subgroups then
 $G$  virtually has $\mathcal{P}$ if and only if $K$ has $\mathcal{P}$.
\end{prop}

\begin{proof}
 Since $K$ has finite index in $G$, the `if' direction is clear. 
 To show the `only if' part, suppose that $G$  virtually has $\mathcal{P}$ and has the congruence subgroup property. 
 Thus there exists some $n$ such that $\st_G(n)$ has $\mathcal{P}$ and therefore $\rst_G(n)$ has $\mathcal{P}$. 
 Since $G$ is regular branch over $K$, we have $\psi_n(\rst_G(n))\geq K\times\cdots\times K$ and therefore $K$ must have $\mathcal{P}$.
\end{proof}

% \begin{prop}
% Let $G$ be a GGS-group with non-constant defining vector $\mathbf{e}$.
% Then:
% \begin{enumerate}
% \item
% If $\mathbf{e}$ is non-symmetric, then $G$ is virtually torsion-free if and only if $G'$ is torsion-free.
% \item
% If $\mathbf{e}$ is symmetric, then $G$ is virtually torsion-free if and only if $\gamma_3(G)$ is torsion-free.
% \end{enumerate}
% \end{prop}
% 
% 
% 
% \begin{proof}
% We only prove (i), since (ii) is completely similar.
% Obviously, it suffices to prove the `only if' part of the statement.
% So let us assume that $G$ is virtually torsion-free.
% Since $G$ satisfies the congruence subgroup property, $\st_G(n)$ is torsion-free for some $n\ge 2$.
% Let us choose $n$ as small as possible.
% If $n\ge 3$ then, by Lemma 3.3 of \cite{Alc}, we have
% \[
% \psi(\st_G(n))
% =
% \st_G(n-1) \times \overset{p}{\cdots} \times \st_G(n-1).
% \]
% Consequently $\st_G(n-1)$ is also torsion-free, which is contrary to the choice of $n$.
% Thus we have $n=2$, and then since $\st_G(1)'\le \st_G(2)$ and
% \[
% \psi(\st_G(1)')=G'\times\overset{p}{\cdots}\times G'
% \]
% by Proposition \ref{branch}, we conclude that $G'$ is torsion-free.
% \end{proof}

As a consequence, a natural strategy in order to answer Barnea's second question in the affirmative is to consider a GGS-group $G$ with non-symmetric defining vector and examine whether $G'$ is torsion-free.
We will show that this is the case, for instance, for the group with defining vector $\mathbf{e}=(1,\ldots,1,\lambda)$ with $\lambda\in\F_p\smallsetminus\{1,2\}$, for every odd prime $p$, although the proof is valid for other vectors too. 

% In particular, to find examples answering Barnea's second question, it will suffice to show that for certain GGS-groups $G$ with non-symmetric defining vector, $G'$ is torsion-free.
% We will show this for the groups with defining vector $\mathbf{e}=(1,\ldots,1,0)$ but the proof is valid for other vectors too as explained below. %(for example (2,1,0,0) )
%
% Thus we will prove that the GGS-group with defining vector $\mathbf{e}=(1,\ldots,1,0)$ has torsion-free derived subgroup.

We need the following two lemmas.
In the remainder, we write $b_i$ for the conjugate $b^{a^i}$ for all $i\in\Z$; observe that $b_i=b_j$ if $i\equiv j\pmod p$. Also, we have $\st_G(1)=\langle b_1,\ldots,b_p \rangle$.

\begin{lem}
Let $G$ be a GGS-group and let $h\in\st_G(1)$.
Then the following conditions are equivalent:
\begin{enumerate}
\item
$h\in G'$.
\item
If $\psi(h)=(h_1,\dots, h_p)$, then $h_1\dots h_p\in G'$.
\item
$\psi((ah)^p)\in G'\times \cdots \times G'$.
\end{enumerate}
\end{lem}

\begin{proof}
Let $\Phi: \st_G(1)\longrightarrow G/G'$ be the homomorphism given by $\Phi(h)=h_1\dots h_pG'$, where
$\psi(h)=(h_1,\dots,h_p)$.
Clearly, we have $\Phi(h^a)=\Phi(h)$ for all $h\in\st_G(1)$, and then $\Phi(b_i)=\Phi(b)$ for all $i\in\Z$.
If we write $h\in\st_G(1)$ in the form $h=b_{i_1}^{r_1}\dots b_{i_k}^{r_k}$, with $r_1,\ldots,r_k\in\Z$, it follows that
$\Phi(h)=\Phi(b)^{r_1+\cdots+r_k}$.
Since $G/G'$ is elementary abelian and $\Phi(b)$ is non-trivial, we have $h_1\ldots h_p\in G'$ if and only if
$r_1+\dots+r_k=0$ in $\F_p$.
Now, by Theorem 2.11 in \cite{Alc}, the latter condition is equivalent to $h\in G'$.
This proves that (i) and (ii) are equivalent.

Now we prove the equivalence between (ii) and (iii).
Since
\[
(ah)^p = h^{a^{p-1}} h^{a^{p-2}} \ldots h^ah,
\]
the $i$th component of $\psi((ah)^p)$ is $ h_{i+1} \ldots h_{i+p-1} h_i$, where the indices are to be reduced modulo $p$ to the interval $[1,p]$, and the result follows.
\end{proof}

\begin{lem}
\label{lem:order p}
Let $G$ be a GGS-group and let $g\in G$ be such that $g^p=1$.
Then $g\in \langle a \rangle G'\cup \langle b \rangle G' \cup G'$.
\end{lem}

\begin{proof}
Suppose for a contradiction that $g=a^rb^sf$, with $f\in G'$ and $r,s\not\equiv 0 \pmod p$.
By considering a suitable power of $g$, we may assume that $r=1$.
Since $\psi(g^p)=(1,\ldots,1)$, it follows from the previous lemma that $b^sf\in G'$, which is a contradiction.
\end{proof}

The following two results complete the proof of Theorem D.
We first require a somewhat technical definition, motivated by the following discussion.
Let $G$ be a GGS-group with defining vector $\mathbf{e}=(e_1,\dots,e_{p-1})$, and let
$g\in G'\smallsetminus \st_G(1)'$.
By Theorems 2.10, 2.11 and 2.14 in \cite{Alc}, we can uniquely write $g=b_1^{i_1}\dots b_p^{i_p}h$ for some
$h\in\st_G(1)'$ and some non-zero vector $(i_1,\dots,i_p)\in \mathbb{F}_p^p$ with $i_1+\dots +i_p=0$.
Since $\psi(\st_G(1)')\subseteq G'\times \overset{p}{\cdots} \times G'$, we have
\begin{equation}
\label{eqn:psi(g)}
\psi(g)
=
\psi(b_1^{i_1}\dots b_p^{i_p}) \psi(h)
=
(a^{m_1}b_{k_1}^{i_1} f_1,\dots,a^{m_p}b_{k_p}^{i_p} f_p),
\end{equation}
for some $k_j\in\{1,\ldots,p\}$ and $f_j\in G'$, and with
\begin{equation}
\label{eqn:mj}
(m_1,\ldots,m_p)=(i_1,\ldots,i_p)C,
\end{equation}
where $C$ is the circulant matrix
\[
C
=
\begin{pmatrix}
0 & e_1 & \cdots & e_{p-1}\\
e_{p-1} & 0 & \cdots & e_{p-2}\\
\vdots & \vdots & \ddots & \vdots\\
e_1 & e_2 &\cdots & 0
\end{pmatrix}.
\]

\begin{defi}
We say that a vector $\mathbf{e}\in\F_p^{p-1}$ satisfies {\em property TF\/} if for every non-zero vector
$(i_1,\dots,i_p)\in \mathbb{F}_p^p$ with $i_1+\cdots+i_p=0$, there exists $j\in \{1,\ldots,p\}$
such that $m_ji_j\neq 0$, where the $m_j$ are defined by (\ref{eqn:mj}).
\end{defi}

\begin{thm}
Let $G$ be the GGS-group defined by a vector $\mathbf{e}=(e_1,\dots,e_{p-1})$ satisfying property TF.
Then $G'$ is torsion-free.
\end{thm}

\begin{proof}
The GGS-group $G$ lies in a Sylow pro-$p$ subgroup of $\Aut T$, and consequently a torsion element must be of
$p$-power order.
Thus it suffices to show that $G'$ has no elements of order $p$.

Let us consider an arbitrary element $g\in G'$.
Assume first that $g\in G'\smallsetminus \st_G(1)'$.
As explained above, $\psi(g)$ satisfies (\ref{eqn:psi(g)}) and, since $\mathbf{e}$ satisfies property TF, there is some
$j\in\{1,\dots,p\}$ such that $m_ji_j\neq 0$.
This, together with Lemma \ref{lem:order p}, implies that the $j$th component of $\psi(g)$ is not of order $p$, and therefore neither is $g$.

Now we assume that $g\in\st_G(1)'$.
Thus we can consider the largest integer $n\ge 0$ for which
\[
\psi_n(g)\in\st_G(1)'\times \overset{p^n}{\cdots} \times \st_G(1)'.
\]
Then $g\in\st_G(n+1)$ and, since $\psi(\st_G(1)')\subseteq G'\times \overset{p}{\cdots} \times G'$, the vector
$\psi_{n+1}(g)$ has a component in $G'\smallsetminus \st_G(1)'$.
By the previous paragraph, $g$ is not of order $p$ also in this case.
\end{proof}

\begin{cor}
Let $G$ be the GGS-group defined by the vector $\mathbf{e}=(1,\dots,1,\lambda)$, with $\lambda\in\F_p$.
Then the following hold:
\begin{enumerate}
\item
If $\lambda\ne 2$ then $G'$ is torsion-free.
\item
If $\lambda\ne 1,2$ then $G'$ is an infinite, finitely generated, residually finite, and torsion-free group whose profinite completion is a pro-$p$ group.
\end{enumerate}
\end{cor}

\begin{proof}
Clearly, we only need to prove (i).
By the previous theorem, it suffices to show that $\mathbf{e}$ satisfies property TF.
Consider then a non-zero vector $(i_1,\dots,i_p)\in \mathbb{F}_p^p$ with $i_1+\cdots+i_p=0$.
If $J$ is the matrix over $\F_p$ with all entries equal to $1$, then $(i_1,\dots,i_p)J=(0,\dots,0)$ and, by (\ref{eqn:mj}), we have
\begin{equation}
\label{eqn:m's and i's}
\begin{split}
(m_1,\dots,m_p)
&=
(i_1,\dots,i_p) (C-J)
\\[5pt]
&=
(i_1,\dots,i_p)
\begin{pmatrix}
-1 & 0 & \cdots &  0 & \lambda-1
\\
\lambda-1 & -1 & \cdots & 0 & 0
\\
\vdots & \vdots &  & \vdots & \vdots
\\
0 & 0 &\cdots & \lambda-1 & -1
\end{pmatrix}.
\end{split}
\end{equation}
Since the last matrix is circulant, the same equality holds for any cyclic permutation of the tuples
$(m_1,\ldots,m_p)$ and $(i_1,\ldots,i_p)$.
Thus we may assume without loss of generality that $i_1\ne 0$.
If $m_1\ne 0$, we are done. 
If not, let $j\geq 2$ be as large as possible subject to the condition $m_1=\cdots=m_{j-1}=0$.
Then by (\ref{eqn:m's and i's}) we have $i_1=(\lambda-1)^{j-1}i_j$.
Consequently $i_j\ne 0$, and we are done if $j\le p$.
Otherwise, if $m_1,\dots,m_p$ are all $0$, we get $i_1=(\lambda-1)^p i_1=(\lambda-1) i_1$.
This is a contradiction, since $i_1\ne 0$ and $\lambda\ne 2$.
\end{proof}

\bibliographystyle{abbrv}
\bibliography{on_the_congruence_subgroup_property_for_GGS-groups}

\end{document}